\documentclass[11pt,a4paper,reqno]{amsart}
\usepackage[utf8]{inputenc}
\usepackage[english]{babel}

\usepackage{lipsum}
\usepackage[arrow]{xy}
\usepackage{graphicx}
\usepackage{amsfonts}
\usepackage{amsmath}
\usepackage{amssymb}
\usepackage{geometry}
\usepackage{amsthm}
\usepackage{enumerate}
\usepackage[dvipsnames]{xcolor}
\usepackage{esint}
\usepackage{tcolorbox}
\definecolor{mygray}{gray}{0.8}
\usepackage{cancel}

\newtheorem{theorem}{Theorem}[section]
\newtheorem{lemma}[theorem]{Lemma}
\newtheorem{corollary}[theorem]{Corollary}
\newtheorem{proposition}[theorem]{Proposition}
\theoremstyle{definition}
\newtheorem{definition}[theorem]{Definition}
\newtheorem{remark}[theorem]{Remark}
\newtheorem{example}[theorem]{Example}

\usepackage{hyperref}
\hypersetup{
	  colorlinks,%
	citecolor=blue,%
	filecolor=red,%
	linkcolor=red,%
	urlcolor=blue
}

\allowdisplaybreaks

\begin{document}
	
	
\title[One-sided Muckenhoupt weights and one-sided weakly porous sets in $\mathbb{R}$]{One-sided Muckenhoupt weights and one-sided weakly porous sets in $\mathbb{R}$}
	
\author[Hugo Aimar]{Hugo Aimar}
\author[Ivana G\'{o}mez]{Ivana G\'{o}mez}
\author[Ignacio G\'{o}mez Vargas]{Ignacio G\'{o}mez Vargas}
\author[Francisco Mart\'{i}n-Reyes]{Francisco Javier Mart\'{i}n-Reyes}
\begin{abstract}
  In this work, we introduce the geometric concept of one-sided weakly porous sets in the real line and show that a set $E\subset\mathbb{R}$ satisfies $d(\cdot,E)^{-\alpha}\in A_1^+(\mathbb{R})\cap L^1_\textrm{loc}(\mathbb{R})$ for some $\alpha>0$ if and only if $E$ is right-sided weakly porous. Furthermore, we find that the property of being both left-sided and right-sided weakly porous is equivalent to the recent weakly porous condition discussed in the bibliography, which, in turn, was previously found to be intimately related to the usual class of Muckenhoupt weights $A_1$.
\end{abstract}

\maketitle

\section{Introduction}
\label{section1}

The class of Muckenhoupt weights in $\mathbb{R}^n$, and even in some more general settings, remains one of the most studied density classes in the field of harmonic analysis due to its applications in the search of bounds for several types of operators and their role in the analysis of the regularity of solutions in PDE, see \cite{GarciaCuervaRubioFranciaBook}, \cite{LOPEZ} and \cite{DYDA}. Despite the mathematical community's interest in this particular topic, less attention has been paid to the relations of Muckenhoupt weights with the geometry of subsets. In this sense, a recent series of articles seems to have shed light on the problem of finding those sets $E$ for which there exists some $\alpha>0$ such that the power $d(\cdot,E)^{-\alpha}$ belong to the $A_1$ Muckenhoupt class in different settings \cite{ANDERSON,MUDARRA,NOSOTROS}. One can think of the origin of this problem as a natural question arising from the well-known fact that $|x|^{-\alpha}$ belongs to $A_p(\mathbb{R}^n)$ for every $-n<\alpha<n(p-1)$ and $1<p<\infty$, where $|x|$ can be seen as the distance function $d(x,\{0\})$.

It is essential to mention that the authors in \cite{ANDERSON} have found a necessary and sufficient condition on a set $E$ to achieve $d(\cdot,E)^{-\alpha}\in A_1(\mathbb{R}^n)$ for some $\alpha>0$ and named it \textit{weak porosity}, in line with the concept first introduced in \cite{VASIN} to study classes of weights and BLO functions in the torus. Roughly speaking, and restricting ourselves to the one-dimensional case, a weakly porous set $E\subset\mathbb{R}$ is a set such that all open intervals $I$ contain a finite quantity of sub-intervals $I_1,\ldots,I_N$ which do not intersect $E$ and whose lengths sum at least a fixed proportion of the number $|I|$, standing for the Lebesgue measure of $I$. Intuitively, one can think of $E$ as a set that is ``full of pores everywhere'' with respect to both the Euclidean distance and the Lebesgue measure, in contrast with other more common definitions of porosity in metric spaces which relate only to the metric and have no relation at all with any existing measure \cite{FALCONER,SHMERKIN}. An interesting property of any weakly porous set $E$, as defined in \cite{ANDERSON}, is the ``isotropic'' distribution of the sizes of pores. More precisely, given three consecutive intervals $I$, $J$, and $K$ of the same length, if $E$ has a pore of length $\rho$ in $J$, then $E$ has at least one pore of comparable size to $\rho$ both in $I$ and $K$.

Sawyer first introduced one-sided Muckenhoupt weights in \cite{SAWYER} and were further characterized in \cite{KICO}, among other works. The basic idea behind its conception lies in their property of inducing measures for which the one-sided versions of the Hardy-Littlewood maximal function become bounded operators. Overall, the one-sided Muckenhoupt classes present many similarities with respect to the usual ``two-sided'' classes $A_p(\mathbb{R})$, but also big differences. It is for this reason that we aimed to determine when powers of distance functions belong to the one-sided Muckenhoupt classes $A^-_1(\mathbb{R})$ or $A^+_1(\mathbb{R})$ by proposing one-sided versions of the geometric concept of weak porosity. We will show that one-sided weakly porous sets share a lot of properties with standard weakly porous sets as given in \cite{ANDERSON} but usually exhibit an anisotropy in its pore sizes distribution in a biased way towards one side, which could explain the intrinsic relationship to be established between them and the condition $d(\cdot,E)^{-\alpha}\in A_1^-(\mathbb{R})$ or $d(\cdot,E)^{-\alpha}\in A_1^+(\mathbb{R})$, as appropriate. Furthermore, after providing the proof of Theorem \ref{theorem4.1}, the main result of this paper, the following consistency diagram can be obtained

%

\begin{figure}[h]
	 \includegraphics[width=\textwidth]{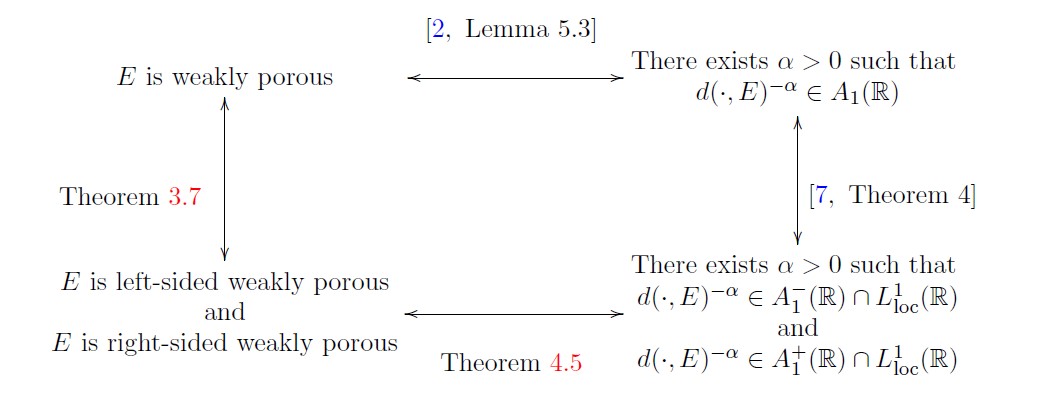}
\end{figure}

\vspace*{0.5cm}
\noindent where each arrow indicates an if and only if equivalence between the statements.

This paper is organized as follows. Section~\ref{section2} is devoted to introduce the basic theory of one-sided Muckenhoupt weights, together with some specific characterization of them to be used in the rest of these pages. In Section~\ref{section3} we present our definition of one-sided weakly porous sets and compare it with the one given by Anderson \textit{et al.} in \cite{ANDERSON}. Finally, in Section~\ref{section4} we establish the relationship between one-sided weights and one-sided weakly porous sets.

\section{One-sided Muckenhoupt weights}
\label{section2}

A function $w:\mathbb{R}\to\mathbb{R}$ is called a weight in $\mathbb{R}$ if it is measurable and non-negative. A particular and largely studied class of weights in $\mathbb{R}$ is the $A_p(\mathbb{R})$ Muckenhoupt class for $1\leq p<\infty$, see \cite{GarciaCuervaRubioFranciaBook}. These weights have the defining property of making the Hardy-Littlewood maximal operator $M$ of weak type $(p,p)$, meaning there is some constant $C>0$ such that
$$w(\{x\in\mathbb{R}:Mf(x)>t\})\leq\frac{C}{t^p}\int_{\mathbb{R}}|f(x)|^pw(x)dx,\text{ for every }f\in L^1_\text{loc}(\mathbb{R}),$$
where $w(A)$ denotes the quantity $\int_Aw(x)dx$ for a given Lebesgue measurable set $A$ and $Mf(x)=\sup_{h>0}\frac{1}{2h}\int_{x-h}^{x+h}|f(y)|dy$, with $f\in L^1_\text{loc}(\mathbb{R})$. The actual definition of Muckenhoupt weights requires them to be locally integrable weights for which there exists a constant $C>0$ such that
\begin{equation}
    \label{eq2.1}
    \Big(\frac{1}{|I|}\int_Iw(x)dx\Big)\Big(\frac{1}{|I|}\int_Iw(x)^{-\frac{1}{p-1}}dx\Big)^{p-1}\leq C,\textrm{ for every interval }I
\end{equation}
in the case $1<p<\infty$, whereas for $p=1$ the $A_1(\mathbb{R})$ condition reads
\begin{equation}
    \label{eq2.2}
    \frac{1}{|I|}\int_Iw(x)dx\leq C\:\textrm{inf ess}_{x\in I}w(x),\textrm{ for every interval }I.
\end{equation}

As discussed in the introduction, all concepts introduced so far have natural counterparts in the one-sided case. Indeed, starting with the maximal operator $M$ we can define new operators $M^-f(x)=\sup_{h>0}\tfrac{1}{h}\int_{x-h}^x|f(y)|dy$ and $M^+f(x)=\sup_{h>0}\tfrac{1}{h}\int_x^{x+h}|f(y)|dy$. On the other hand, right-sided Muckenhoupt weights for $1<p<\infty$ are introduced by replacing \eqref{eq2.1} with
\begin{equation}
    \label{eq2.3}
    \sup_{h>0}\Big(\frac{1}{h}\int_{x-h}^xw(y)dy\Big)\Big(\frac{1}{h}\int_x^{x+h}w(y)^{-\frac{1}{p-1}}dy\Big)^{p-1}\leq C,\textrm{ for every }x\in\mathbb{R}
\end{equation}
and \eqref{eq2.2} with
\begin{equation}
    \label{eq2.4}
    M^-w(x)\leq Cw(x),\textrm{ for almost every }x\in\mathbb{R}
\end{equation}
in the case $p=1$. Right-sided Muckenhoupt weight classes are denoted as $A_p^+(\mathbb{R})$. Left-sided Muckenhoupt weight classes $A_p^-(\mathbb{R})$ are defined analogously.

\begin{remark}
    \label{rem1}
    Regular and one-sided Muckenhoupt classes are related by the formula $A_p(\mathbb{R})=A_p^-(\mathbb{R})\cap A_p^+(\mathbb{R})$ (see \cite[Theorem~4]{KICO}).
\end{remark}

Given $w\in A_1^+(\mathbb{R})$, it is possible to find points $-\infty\leq x_w^0\leq x_w^1\leq\infty$ such that
    \begin{equation}
        \begin{cases}
        \label{eq2.5}
       w(x) = 0 & x\in(-\infty,x_w^0],\\
       0<w(x)<\infty &x\in (x_w^0,x_w^1), \\
       w(x) = \infty &x\in [x_w^1,\infty). \end{cases}
       \end{equation}
Namely, set $x_w^0=\sup\{x\in\mathbb{R}:\int_{-\infty}^xw(y)dy=0\}$ and $x_w^1=\text{inf ess}\{x\in\mathbb{R}:w(x)=\infty\}$. In particular, $w$ is locally integrable if and only if $x_w^1=\infty$. A certainly useful and alternative characterization of functions in $A_1^+(\mathbb{R})$ (which will be used extensively in the next sections) is presented in item ii of the following proposition. Before this, we introduce the notation $I^-:=I\cap[a,\frac{a+b}{2})$ and $I^+:= I\setminus I^-$ in reference to an arbitrary interval $I$ with endpoints $a<b$.

\begin{proposition}
    Given a weight $w$ in $\mathbb{R}$, the following statements are equivalent.
    \begin{enumerate}[i.]
        \item $w$ belongs to $A_1^+(\mathbb{R})$.
        \item  There exists a constant $C>0$ such that
        \begin{equation}
            \label{eq2.0}
            \frac{1}{|I^-|}\int_{I^-}w(y)dy\leq C\:\textrm{ess inf}_{x\in I^+}w(x),\text{ for every open interval }I\subset\mathbb{R}.
        \end{equation}
        \item There exists a constant $C>0$ such that for every triple $a<b<c$ we have
        $$\frac{1}{c-a}\int_a^bw(y)dy\leq C\:\textrm{ess inf}_{x\in(b,c)}w(x).$$
    \end{enumerate}
\end{proposition}

\begin{proof}
    \noindent ($i\implies ii$) Assume $w$ satisfies \eqref{eq2.4} with constant $C>0$. Let $A=\{x\in\mathbb{R}:M^-w(x)>C\:w(x)\}$ and set $\tilde{w}(x)=w(x)$ for $x\notin A$ and $\tilde{w}(x)=\infty$ otherwise. Clearly, $w(x)=\tilde{w}(x)$ a.e., $M^-w(x)=M^-\tilde{w}(x)$ and $M^-\tilde{w}(x)\leq C\tilde{w}(x)$ for every $x\in\mathbb{R}$. Thus, $\tilde{w}$ satisfies the $A_1^+(\mathbb{R})$ condition on every point with constant $C$ and, if $I$ is an open interval with $|I|=h$, it follows that for every $x\in I^+$
    \begin{align*}
        \frac{1}{|I^-|}\int_{I^-}w(y)dy=\frac{1}{|I^-|}\int_{I^-}\tilde{w}(y)dy =\frac{2}{h}\int_{I^-}\tilde{w}(y)dy \leq\frac{2}{h}\int_{x-h}^x\tilde{w}(y)dy\leq 2C\Tilde{w}(x).
    \end{align*}
    Consequently
    \begin{align*}
        \frac{1}{|I^-|}\int_{I^-}w(y)dy\leq 2C\:\text{inf ess}_{x\in I^+}\tilde{w}(x)= 2C\:\text{inf ess}_{x\in I^+}w(x)
    \end{align*}
    and the result follows by renaming $C$ appropriately.

    \noindent ($ii\implies iii$) Suppose first that $c-b\leq b-a$. If we set $I=(a,2b-a)$, then $I^-=(a,b)$ and $(b,c)\subset I^+$, thus
    \begin{align*}
        \frac{1}{c-a}\int_a^b w(y)dy &= \frac{|I^-|}{c-a}\frac{1}{|I^-|}\int_{I^-}w(y)dy \\
        &\leq C\frac{|I^-|}{c-a}\text{ess inf}_{x\in I^+}w(x)
        \leq C\:\text{ess inf}_{x\in I^+}w(x)
        \leq C\:\text{ess inf}_{x\in (b,c)}w(x).
    \end{align*}
    Consider now the case $c-b>b-a$. Take $x_0=c$ and define the sequence $\{x_n\}_{n\in\mathbb{N}}$ by setting $x_n=\frac{a+x_{n-1}}{2}$ for each $n\in\mathbb{N}$. There exists then some natural number $N$ such that $x_{N+1}\leq b<x_N<x_{N-1}<\ldots<x_0=c.$ Let us denote, for convenience, $y_n=x_n$ for $0\leq n\leq N$ and $y_{N+1}=b$. We have that $\text{inf ess}_{x\in(b,c)}w(x)=\text{inf ess}_{x\in(y_k,y_{k-1})}w(x)$, for some $k\in\{1,\ldots,N+1\}$. Also we note that $y_{k-1}-y_k\leq a-y_k$, so if $I$ is an interval such that $I^-=(a,y_k)$, then $(y_{k-1},y_k)\subset I^+$ and we can write
    \begin{align*}
        \frac{1}{c-a}\int_a^bw(y)dy &\leq \frac{1}{c-a}\int_a^{y_k}w(y)dy \\ 
        &\leq \frac{1}{y_k-a}\int_a^{y_k}w(y)dy\leq C\:\text{inf ess}_{x\in(y_{k-1},y_k)}w(x)
        = C\:\text{inf ess}_{x\in(b,c)}w(x).
    \end{align*}

    \noindent ($iii\implies i$) Given $N\in\mathbb{N}$, take
    $w_N(x)= 
      \min\{w(x),N\}$. Using the estimates $w_N(x)\leq w(x)$ and $w_N(x)\leq N$ it is easy to see that $\frac{1}{c-a}\int_a^bw_N(y)dy\leq \tilde{C}\min\{N,\text{inf ess}_{x\in(b,c)}w(x)\}$ for every triple $a<b<c$, where the constant  $\tilde{C}=\min\{1,C\}$. If $b$ and $c$ are such that $\text{inf ess}_{x\in(b,c)}w(x)\geq N$ and we take $A$ to be the set $\{x\in(b,c):w(x)<N\}$, then $|A|=0$.  This means that
    \begin{align*}
        \frac{1}{c-a}\int_a^bw_N(y)dy \leq \tilde{C}N= \tilde{C}\:\text{inf ess}_{x\in(b,c)\setminus A}w_N(x)= \tilde{C}\:\text{inf ess}_{x\in(b,c)}w_N(x).
    \end{align*}
    \noindent On the other hand, if $\text{inf ess}_{x\in(b,c)}w(x)<N$, then \begin{align*}
        \frac{1}{c-a}\int_a^bw_N(y)dy &\leq \tilde{C}\:\text{inf ess}_{x\in(b,c)}w(x)= \tilde{C}\:\text{inf ess}_{x\in(b,c)}w_N(x).
    \end{align*}
    We have seen that $w_N$ satisfies item iii with the same constant $\tilde{C}$ as $w$, independently from the value of $N$. Furthermore, since it is bounded, $w_N\in L^1_\text{loc}(\mathbb{R})$ and $\frac{1}{h}\int_x^{x+h}w_N(y)dy\to w_N(x)\text{ as $h\to0^+$}$ for almost every $x$ by the Lebesgue Differentiation Theorem. In consequence, if $x>a$ and $h>0$,
    \begin{align*}
        \frac{1}{x+h-a}\int_a^x w_N(y)dy\leq \tilde{C}\:\text{inf ess}_{y\in(x,x+h)}w_N(y)\leq \frac{\tilde{C}}{h}\int_x^{x+h}w_N(y)dy.
    \end{align*}
    \noindent By taking the limit as $h\to0^+$ we get $\frac{1}{x-a}\int_a^x w_N(y)dy \leq\tilde{C}\:w_N(x),$ and then by taking the limit once more as $N\to\infty$ and by monotone convergence we get $\frac{1}{x-a}\int_a^x w(y)dy \leq \tilde{C}\:w(x)$, which implies i.
\end{proof}



\section{Weak porosity and one-sided weak porosity in $\mathbb{R}$}
\label{section3}

As already mentioned, the concept of weak porosity has already been studied in various settings such as Euclidean spaces \cite{ANDERSON}, metric spaces with doubling measures \cite{MUDARRA}, and spaces of homogeneous type \cite{NOSOTROS} with minor variations in its phrasing. Definitions to be introduced here were chosen to be in analogy with the ones in the latter work, which encompasses the more general case. Recall that, for an arbitrary interval $I$ with endpoints $a<b$, we denote its left and right splits as $I^-=I\cap[a,\frac{a+b}{2})$ and $I^+= I\setminus I^-$.

\begin{definition}
    \label{def3.1}
    Let $E\subset\mathbb{R}$ be a non-empty set and $I$ be an arbitrary interval with endpoints $a<b$. We denote
    $$\Lambda(I)=\{s>0:\exists\: y\in I\text{ such that }(y-s,y+s)\subset I\setminus E\}.$$
    We now introduce the $E$\textbf{-free maximal hole function} as $\rho_E(I):=\sup \Lambda(I)$. In case $\Lambda(I)=\emptyset$, we set $\rho_E(I)=0$.
\end{definition}

When the context is clear, we drop the subindex $E$ from the maximal hole function $\rho$ associated with this set. Intuitively, the function $\rho$ applied to $I$ returns the ``radius'' of the greatest pore in $I$ with respect to $E$. It is not hard to realize that the supremum used in the definition of $\rho$ can actually be replaced by a maximum.

\begin{remark}
    A common definition of porous sets $E$ in a metric space $(X,d)$ requires the existence of a constant $0<\beta(E)<1$ such that for every ball $B(x,r)$ in $X$, there exists some $y\in X$ which satisfies $B(y,\beta r)\subset B(x,r)\setminus E$. In the context of $\mathbb{R}$ and having Definition~\ref{def3.1} in mind, this could be rephrased as the requirement that $\rho_E(I)\geq\frac{\beta}{2}|I|$ for every interval $I\subset\mathbb{R}$ and some $0<\beta(E)<1$. 
\end{remark}

\begin{definition}
    \label{def3.2}
    We say that a given non-empty set $E\subset\mathbb{R}$ is $\boldsymbol{(\sigma,\gamma)}$\textbf{-weakly porous} (or $(\sigma,\gamma)$-w.p. for short) if there are constants $\sigma,\gamma\in(0,1)$ such that for every open interval $I$ we can find a collection of pairwise disjoint open intervals $I_1,\ldots,I_N$ (where $N=N(I)$) such that
\begin{itemize}
    \item[\textit{i.}] $I_i\subset I\setminus E$ for every $i=1,\ldots,N$;
    \item[\textit{ii.}] $|I_i|\geq2\gamma\rho(I)$; 
    \item[\textit{iii.}] $\sum_{i=1}^N |I_i|\geq \sigma|I|$.
\end{itemize}
\end{definition}

Thanks to the previous remark, one can quickly recognize that weak porosity is, in fact, a weaker condition than usual porosity. A few changes in the definition are enough to obtain an appropriate right-sided version of this concept.

\begin{definition}
    \label{def3.3}
    We say that a given non-empty set $E\subset\mathbb{R}$ is $\boldsymbol{(\sigma,\gamma)}$\textbf{-right-sided weakly porous} if there are constants $\sigma,\gamma\in(0,1)$ such that for every open interval $I$ we can find a collection of pairwise disjoint open intervals $I_1,\ldots,I_N$ (where $N=N(I)$) such that
\begin{itemize}
    \item[\textit{i.}] $I_i\subset I^-\setminus E$ for every $i=1,\ldots,N$;
    \item[\textit{ii.}] $|I_i|\geq2\gamma\rho(I^+)$; 
    \item[\textit{iii.}] $\sum_{i=1}^N |I_i|\geq \sigma|I^-|$.
\end{itemize}
\end{definition}

An immediate consequence of a set $E$ being $(\sigma,\gamma)$-right-sided weakly porous is that $\rho(I^-)\geq\gamma\rho(I^+)$ for every open interval $I$. In analogy to Definition \ref{def3.3}, a set $E\in\mathcal{P}(\mathbb{R})\setminus\{\emptyset\}$ is said to be $\boldsymbol{(\sigma,\gamma)}$\textbf{-left-sided weakly porous} if $E$ satisfies the previous conditions replacing i-iii with
\begin{itemize}
    \item[\textit{i'.}] $I_i\subset I^+\setminus E$ for every $i=1,\ldots,N$;
    \item[\textit{ii'.}] $|I_i|\geq2\gamma\rho(I^-)$; 
    \item[\textit{iii'.}] $\sum_{i=1}^N |I_i|\geq \sigma|I^+|$.
\end{itemize}

We write $(\sigma,\gamma,+)$-w.p. and $(\sigma,\gamma,-)$-w.p. to refer to $(\sigma,\gamma)$-right-sided weakly porous sets and $(\sigma,\gamma)$-left-sided weakly porous sets, respectively. Also, we will sometimes omit any reference to the constants $\sigma$ and $\gamma$ to lighten up the notation. Before dwelling on any examples, we aim to characterize the family of sets complying with Definition~\ref{def3.3} with some basic results. At this point, it is worthy to mention that statements to be made from now on will generally focus on right-sided weakly porous sets, but can easily be rewritten in terms of left-sided ones by analogy.

\begin{proposition}
    \label{prop3.1}
    Let $E\subset\mathbb{R}$ be a non-empty set.
    \begin{enumerate}[i.]
        \item $E$ is $(\sigma,\gamma,+)$-w.p. iff $\bar{E}$ is $(\sigma,\gamma,+)$-w.p.
        \item If $E$ is right-sided w.p., then $|E|=0$.
    \end{enumerate}
\end{proposition}

\begin{proof}
    Item \textit{i.} follows from the fact that a given open interval $J$ satisfies $J\cap E=\emptyset\iff J\cap\bar{E}=\emptyset$. By \textit{i.} and because of the completeness of the Lebesgue measure, to prove \textit{ii.} we can assume that $E$ is a closed set. Take $I$ to be any open interval and set $J$ as the open interval such that $J^-=I$. Using the right-sided weakly porous condition of $E$ applied to $J$, we can find a sequence of pairwise disjoint intervals $I_1,\ldots,I_N$ contained in $I$ such that $\sigma|I|\leq\sum_{i=1}^N|I_i|\leq|I\setminus E|$. Since $E$ is Lebesgue measurable, we have that $|I|=|I\cap E|+|I\setminus E|$, so from the previous inequalities we may write
    $$|I\cap E|\leq \Big(\frac{1}{\sigma}-1\Big)|I\setminus E|=\Big(\frac{1}{\sigma}-1\Big)(|I|-|I\cap E|).$$
    Setting $\lambda:=\sigma^{-1}-1$ and solving for $|I\cap E|$ we get $|I\cap E|\leq \frac{\lambda}{\lambda+1}|I|$. This inequality holds for every open interval $I$, so assuming without loss of generality that $|E|<\infty$ we can estimate its Lebesgue measure in the following manner. Take $\{I_j\}_{j=1}^\infty$ to be a collection of open intervals such that $E\subset \bigcup_{j=1}^\infty I_j$. Then
    \begin{align*}
        |E| = \Big| E\cap \bigcup_{j=1}^\infty I_j \Big| = \Big|  \bigcup_{j=1}^\infty (E\cap I_j) \Big| \leq \sum_{j=1}^\infty |E\cap I_j| \leq \frac{\lambda}{\lambda+1} \sum_{j=1}^\infty|I_j|,
    \end{align*}
    which implies $|E|\leq\frac{\lambda}{\lambda+1}|E|$ by taking the infimum over all such sequences $\{I_j\}_{j=1}^\infty$ and, since $\frac{\lambda}{\lambda+1}<1$, it follows $|E|=0$.
\end{proof}

Just as the usual and one-sided Muckenhoupt classes are related as stated in Remark~\ref{rem1}, we now aim to verify that a set $E$ is weakly porous if and only if it is one-sided weakly porous, in both possible directions. Before that, we must borrow the following result concerning the doubling-like condition satisfied by the maximal hole function associated to any weakly porous set.

\begin{lemma}[\cite{ANDERSON}, Lemma 3.2 (ii)]
    \label{lemma3.0}
    Let $E$ be a $(\sigma,\gamma)$-w.p. set. If $J\subset I$ are two open intervals satisfying $|I|=2|J|$, then there exists some constant $\Phi=\Phi(\sigma,\gamma)>1$ independent from both $I$ and $J$ such that
$$\rho(I)\leq \Phi\rho(J).$$
\end{lemma}

\begin{theorem}
    \label{prop3.2}
    Given $E\in\mathcal{P}(\mathbb{R})\setminus\{\emptyset\}$, the following statements are equivalent.
    \begin{enumerate}[a.]
        \item There are constants $\sigma,\gamma\in(0,1)$ such that $E$ is $(\sigma,\gamma)$-w.p.
        \item There are constants $\sigma_0,\gamma_0\in(0,1)$ such that $E$ is $(\sigma_0,\gamma_0,-)$-w.p. and $(\sigma_0,\gamma_0,+)$-w.p. simultaneously.
    \end{enumerate}
\end{theorem}

\begin{proof}
    Assume first that $E$ is $(\sigma,\gamma)$-w.p. for some $0<\sigma,\gamma<1$. Note that in order to prove  $a.$ implies $b.$, it suffices to check the right side porosity condition on $E$. To this end, we pick an open interval $I$ and use the hypothesis applied to $I^-$ to find pairwise disjoint intervals $I_1,\ldots,I_N$ such that $I_i\subset I^-\setminus E$, $|I_i|\geq2\gamma\rho(I^-)$ and $\sum_{i=1}^N |I_i|\geq \sigma|I^-|$. This already accounts for properties \textit{i} and \textit{iii} of Definition~\ref{def3.3} if we set $\sigma_0:=\sigma$, while property \textit{ii} with $\gamma_0:=\gamma\Phi^{-1}$ follows from Lemma~\ref{lemma3.0} since
    \begin{equation*}
    	|I_i|\geq2\gamma\rho(I^-)\geq2\gamma\Phi^{-1}\rho(I)\geq2\gamma\Phi^{-1}\rho(I^+).
    \end{equation*}
    On the other hand, if $E$ is both $(\sigma_0,\gamma_0,-)$-w.p. and $(\sigma_0,\gamma_0,+)$-w.p., notice that $\rho(I)\leq\rho(I^-)+\rho(I^+)\leq 2\max\{\rho(I^-),\rho(I^+)\}$ for any open interval $I$. Then, for $I$ fixed we can assume without loss of generality that $\rho(I^+)\geq\rho(I^-)$ and, therefore, $\rho(I^+)\geq\frac{1}{2}\rho(I)$. The condition $E$ is $(\sigma_0,\gamma_0,+)$-w.p. applied to $I$ then assures the existence of pairwise disjoint intervals $I_1,\ldots,I_N$ such that $I_i\subset I^-\setminus E$, $|I_i|\geq2\gamma_0\rho(I^+)\geq2\gamma\rho(I)$ and $\sum_{i=1}^N |I_i|\geq \sigma_0|I^-|=\sigma|I|$ if we happen to choose $\gamma=\frac{1}{2}\gamma_0$ and $\sigma=\frac{1}{2}\sigma_0$.
\end{proof}

\begin{proposition}
    \label{prop3.3}
    Given a non-empty set $E\subset \mathbb{R}$ and writing $R:\mathbb{R}\to\mathbb{R}$ for the reflection through the origin, i.e. $R(x)=-x$, we have
    \begin{enumerate}[(a)]
        \item if $E$ is $(\sigma,\gamma)$-w.p. then $E\cap[x_0,\infty)$ is $(\sigma_0,\gamma_0,+)$-w.p. for every $x_0\in\mathbb{R}$;

        \item if $E$ is $(\sigma,\gamma,\pm)$-w.p. then $R(E)$ is $(\sigma,\gamma,\mp)$-w.p.
    \end{enumerate}
\end{proposition}

\begin{proof} To prove $(a)$, consider the set $E_+:=E\cap[x_0,\infty)$ and take $I$ to be an open interval. If $I\subset(-\infty,x_0)$, clearly $E_+$ meets the right-sided w.p. condition on $I$ for any pair of numbers $\sigma,\gamma\in(0,1)$. If instead $I\subset(x_0,\infty)$, then $E_+$ satisfy the w.p. condition on $I$ with parameters $\sigma$ and $\gamma$, and therefore also satisfies the right-sided w.p. condition on $I$ with parameters $\sigma_0:=\sigma$ and $\gamma_0:=\gamma\Phi^{-1}$ as in the proof of Theorem~\ref{prop3.2}. Finally, consider the case where $x_0\in I$ and let $z$ denote the center of this interval. If $z\leq x_0$, then $I^-\cap E_+=\emptyset$ and since $|I^-|\geq2\gamma_0\rho(I^+)$, then the collection $\{I^-\}$ satisfies properties \textit{i}-\textit{iii} of Definition~\ref{def3.3} replacing $E$ with $E_+$. If $z>x_0$, recalling that $E$ is ($\sigma_0,\gamma_0$,+)-w.p. by Theorem \ref{prop3.2}, we can find a collection of pairwise disjoint open intervals $I_1,\ldots,I_N$ satisfying \textit{i}-\textit{iii} of Definition~\ref{def3.3}, but since $E_+\subset E$, this collection also satisfies the mentioned definition for $E_+$ in place of $E$. The proof of $(b)$ becomes straightforward after noticing that $\rho_{R(E)}(I^+)=\rho_E(R(I)^-)$ and that an open interval J satisfies $J\subset I^-$ if and only if $R(J)\subset R(I)^+$.
\end{proof}

\begin{example}
    \label{example3.1}
    It is not difficult to see that the set $\mathbb{Z}$ is $(\frac{1}{2},\frac{1}{2})$-weakly porous, however, its subset $\mathbb{N}_0$ is not weakly porous as the intervals $I_n=(-2^n,2^n)$ and $J_n=(0,2^n)$ satisfies the hypotheses of Lemma~\ref{lemma3.0} for each $n$ but $2^{n-1}=\rho(I_n)\leq\Phi\rho(J_n)=\Phi$ cannot be true for any chosen constant $\Phi$. This example was actually used in \cite{ANDERSON} to show that the weak porosity property is not preserved by set inclusion. Despite this, $\mathbb{N}_0$ is right-sided weakly porous by a direct application of Proposition~\ref{prop3.3}(a).
\end{example}

As in the example above, every cut-off of a weakly porous set above some point $x_0$ is right-sided weakly porous, but not every right-sided weakly porous set that is not left-sided weakly porous can be obtained in this way. 

\begin{example}
    \label{example3.2}
    Take $E$ to be the set $\{z_k\}_{k\in\mathbb{Z}}$ given by
    $$z_k:=\begin{cases} 
      -2^{-k}, &k<0 \\
      k, &k\geq 0. 
   \end{cases}$$
    Then, picking $I_n=(-2^n,2^n)$ and $J_n=(0,2^n)$ as in Example~\ref{example3.1}, we see that there can be no constant $\Phi$ such that $2^{n-2}=\rho(I_n)\leq\Phi\rho(J_n)=\Phi$. Thus, $E$ is not weakly porous. However, it can be shown that $E$ does satisfy the right-sided weakly porous condition with parameters $\sigma=\gamma=\frac{1}{2}$. 
\end{example}

The next result is an analogous of Lemma~\ref{lemma3.0} for the right-sided case and, together with Corollary~\ref{coro3.1}, aims to show how the maximal hole size on a given interval gives information about the pore radii in regions located to the left of that same interval.

\begin{lemma}
    \label{lemma3.1}
    Let $E\subset\mathbb{R}$ be $(\sigma,\gamma,+)$-w.p. and let $I$ be any open interval. Then, \begin{equation}
        \label{eq3.1}
        \rho(I)\leq\frac{\gamma+1}{\gamma}\rho(I^{-}).
    \end{equation}
\end{lemma}

\begin{proof}
     Take $(c,d)$ to be an $E$-free interval contained in $I$ with $d-c=2\rho(I)$, whose existence is assured by the definition of $\rho(I)$. If $(c,d)\subset I^{-}$, then $\rho(I^{-})=\rho(I)$; whereas if $(c,d)\subset I^{+}$, then $\rho(I^{-})\geq\gamma\rho(I^{+})=\frac{1}{2}\gamma(d-c)=\gamma\rho(I)$. So \eqref{eq3.1} holds in both scenarios. Let $z_I$ denote the center of $I$ and consider now the case where $c<z_I<d$. Note that $(c,z_I)\subset I^{-}$ and $(z_I,d)\subset I^{+}$. Thus
    \begin{equation}
        \label{eq3.2}
        \rho(I^{-}) \geq \frac{1}{2}(z_I-c)
    \end{equation}
    and
    \begin{equation}
        \label{eq3.3}
        \rho(I^{-}) \geq \gamma \rho(I^{+}) \geq \frac{1}{2}\gamma (d-z_I).
    \end{equation}
    Combining \eqref{eq3.2} and \eqref{eq3.3} we find that $\rho(I^{-})\geq f(z_I)$, where $f(z):=\max\Big\{\frac{(z-c)}{(d-c)},\gamma\frac{(d-z)}{(d-c)}\Big\}\rho(I)$. As the first argument in the maximum inside $f$ is an increasing function of $z$ and the second one is decreasing, the minimum of $f$ for $z\in(c,d)$ is reached when both entries become equal. This is, when
    $$\frac{(z_\text{min}-c)}{(d-c)}=\gamma\frac{(d-z_\text{min})}{(d-c)}\iff z_\text{min}=\frac{\gamma d+c}{(1+\gamma)},$$
    from what we finally get $\rho(I^{-})\geq f(z_\text{min})=\frac{\gamma}{1+\gamma}\rho(I)$.
\end{proof}

\begin{corollary}
    \label{coro3.1}
    Let $E\subset\mathbb{R}$ be $(\sigma,\gamma,+)$-w.p. and let $I$ be an open interval with center $z_I$. If $J$ is another open interval contained in $I$ and which center $z_J$ satisfies $z_J\leq z_I$, then there exists constants $\theta_1,\theta_2>0$ dependent of $\gamma$ such that
        \begin{equation}
            \label{eq3.4}
            \rho(I^+)\leq \theta_1\Big(\frac{|I|}{|J|}\Big)^{\theta_2}\rho(J^+).
        \end{equation}
\end{corollary}

\begin{proof}
    Let $m$ be the least positive integer such that $2^{m-1}|J|\geq|I|$ (iff $m-1<1+\log_2(\frac{|I|}{|J|})\leq m$). For each $0\leq n\leq m$, take $J_n$ to be the open interval with center $z_J$ and such that $|J_n|=2^n|J|$. In particular, $J_0=J$, $I\subset J_m$ and, since $z_J\leq z_I$, $I^+\subset J^+_m$. Clearly, this last inclusion implies $\rho(I^+)\leq\rho(J_m^+)$. Furthermore, because $\rho(J_n^+)=\rho(J_{n+1}^{+-})$, we can repeatedly apply Lemma~\ref{lemma3.1} to find
    \begin{align*}
        \rho(J_m^+) \leq \frac{\gamma+1}{\gamma} \rho(J_{m-1}^+) \leq \ldots \leq \Big(\frac{\gamma+1}{\gamma}\Big)^{m} \rho(J_0^+) = \Big(\frac{\gamma+1}{\gamma}\Big)^{m} \rho(J^+).
    \end{align*}
    Thus, as $\frac{\gamma+1}{\gamma}>1$ and $m<2+\log_2(\frac{|I|}{|J|})$, we have
    \begin{align*}
        \rho(I^+) \leq \rho(J_m^+) \leq \Big(\frac{\gamma+1}{\gamma}\Big)^{m} \rho(J^+) &\leq \Big(\frac{\gamma+1}{\gamma}\Big)^{2+\log_2(\frac{|I|}{|J|})} \rho(J^+) \\
        &= \Big(\frac{\gamma+1}{\gamma}\Big)^2\Big(\frac{|I|}{|J|}\Big)^{\log_2(\frac{\gamma+1}{\gamma})} \rho(J^+),
    \end{align*}
    so it suffices to take $\theta_1:=(\frac{\gamma+1}{\gamma})^2$ and $\theta_2:=\log_2(\frac{\gamma+1}{\gamma})$.
\end{proof}    

The hypothesis $z_J\leq z_I$ in Corollary~\ref{coro3.1} is essential, as \eqref{eq3.4} does not generally hold without this restriction. Indeed, $\mathbb{N}_0$ is right-sided weakly porous as stated in Example~\ref{example3.1}, but if we take $I_n=(-2n(1+t),2n(1-t))$ for some fixed $0<t<\frac{1}{2}$ and $J_n=(-n,n)$, then \eqref{eq3.4} applied to $I=I_n$ and $J=J_n$ should read
$$\max\{1,2nt\}=\rho(I_n^+)\leq\theta_12^{-\theta_2},$$
which is a contradiction since $\rho(I^+_n)$ tends to infinity as $n\to\infty$ independently from the choice of $t$ and despite having $J_n\subset I_n$ (but $z_J>z_I$). So, a region located right to a relatively big $E$-free hole does not necessarily present similar-sized pores. This should be compared with the case of weights $w\in A_1^+(\mathbb{R})$ in the following sense. The mean value of $w$ on a region located left to an interval $I$ such that $\text{ess inf}_I\:w$ is relatively small must be low enough for \eqref{eq2.0} to hold, whereas that cannot be assured for the mean value of $w$ on regions located right to that same interval $I$. The strong connection between one-sided weakly porous sets and $A_1^+$ weights will be addressed in the next pages.

\section{Relation between one-sided porous sets $E$ and the condition $d(\cdot,E)^{-\alpha}\in A_1^+(\mathbb{R})\cap L^1_\text{loc}(\mathbb{R})$}
\label{section4}

The main goal of this section relies on proving Theorem~\ref{theorem4.1}, which characterizes the collection of sets $E$ for which $d(\cdot,E)^{-\alpha}\in A_1^+(\mathbb{R})\cap L^1_\text{loc}(\mathbb{R})$ for some $\alpha>0$ as the family of all right-sided weakly porous sets, where $d(x,E):=\inf_{y\in E}d(x,y)$. Two technical lemmas will be needed beforehand to reach this objective. In particular, Lemma~\ref{lemma4.2} is very interesting in itself as it reveals an exponential decay in the pore sizes distribution of sets satisfying Definition~\ref{def3.3}; which will be crucial in the proof of the main result.

\begin{proposition}
    \label{lemma4.0}
    Let $E$ be a closed non-empty subset of $\mathbb{R}$ with zero Lebesgue measure. Then, for every open interval $I$ 
we have that
    \begin{equation}
        \label{eq4.0}
       d(x,E)\leq 2\left(1+\frac{d(I,E)}{|I|}\right)\rho(I),\:\textrm{for every } x\in I,
    \end{equation}
    where $d(I,E):=\inf_{x\in I,y\in E} d(x,y)$.
\end{proposition}

\begin{proof}
    We fix an open interval $I$ and assume first that $I\cap E=\emptyset$. In this case, given points $x,y\in I$ and $e\in E$, the triangular inequality implies that $d(x,e)\leq d(x,y)+d(y,e)\leq|I|+d(y,e)$, so it follows $d(x,E)\leq|I|+d(I,E)=\left(1+\frac{d(I,E)}{|I|}\right)|I|$ for every $x\in I$.  Since $\rho(I)=\frac{1}{2}|I|$, we find that $d(x,E)\leq2(1+\frac{d(I,E)}{|I|})\rho(I),\:\forall x\in I$. Now, if $I\cap E\neq\emptyset$, let $x\in I\setminus E$. Since $E$ is closed, $x$ must be contained in a connected component of $I\setminus E$ of the form $(a,b)$. We have that $(a,b)\neq I$ as $I\cap E\neq\emptyset$, so either $a\in E$ or $b\in E$. In any case, we have that $d(x,E)\leq b-a\leq2\rho(I)$, so \eqref{eq4.0} holds.
\end{proof}

\begin{lemma}
    \label{lemma4.1}
    Let $E$ be a closed non-empty subset of $\mathbb{R}$ with zero Lebesgue measure and fix $\eta>0$. Then, there exists a constant $C_0=C_0(\eta)>0$ such that for every open interval $I$ for which $d(I,E)\leq\eta|I|$, we have that
    \begin{equation}
        \label{eq4.1}
        \rho(I^+)\geq C_0\:d(x,E),\:\text{for every } x\in I^+.
    \end{equation}
    In particular, since $|E|=0$, $\rho(I^+)^{-\alpha}\leq C(\alpha,\eta)\:\text{ess inf}_{x\in I^+}d(x,E)^{-\alpha}$ for every $\alpha>0$.
\end{lemma}

\begin{proof}
    Observe first that getting \eqref{eq4.1} is equivalent to showing that $\rho(\mathring{I^+})\geq C_0\:d(x,E)$ for every $x\in\mathring{I^+}$, where $\mathring{I^+}$ denotes the interior of $I^+$. Applying Proposition~\ref{lemma4.0} to the open interval $\mathring{I^+}$, we get the inequality
    \begin{align*}
        d(x,E) &\leq 2\Big(1+\frac{d(\mathring{I^+},E)}{|\mathring{I^+}|}\Big)\rho(\mathring{I^+})\\
        &\leq(6+4\eta)\rho(\mathring{I^+})
        \end{align*}
        valid for every $x\in\mathring{I^+}$.
\end{proof}

\begin{lemma}
    \label{lemma4.2}
    Let $E$ be a closed $(\sigma,\gamma,+)$-w.p. set and let $I=(a,b)$. Also, assume that $I^-\cap E\neq\emptyset$ and that $(c,d)\subset I^+\setminus E$ satisfies $d-c=2\rho(I^+)$. Denote by $\tilde{I}=(a,\frac{c+d}{2})$. Then, there exist constants $0<\beta_1,\beta_2<1$ such that
    $$|F(\beta_1\varepsilon)|\leq\beta_2|F(\varepsilon)|$$
    uniformly for every $0<\varepsilon<\rho(I^+)$, where $F(\varepsilon):=E(\varepsilon)\cap\tilde{I}$ and $E(\varepsilon):=\{x\in \mathbb{R}:d(x,E)<\varepsilon\}$.
\end{lemma}

\begin{proof}
    We begin by fixing $\varepsilon\in(0,\rho(I^+))$ and writing $E(\varepsilon)=\bigcup_{i\in\Gamma}(\tilde{a}_i,\tilde{b}_i)$, where $(\tilde{a}_i,\tilde{b}_i)$ are the at most countable connected components of $E(\varepsilon)$. Notice that, for every $i\in\Gamma$, there exists $e_i\in E$ with $\tilde{a}_i<e_i<\tilde{b}_i$, so $\tilde{b}_i-\tilde{a}_i\geq2\varepsilon$. Now, $F(\varepsilon)=\bigcup_{j=1}^K(a_j,b_j),$
    where each $(a_j,b_j)=(\tilde{a}_i,\tilde{b}_i)\cap\tilde{I}$ for some $i\in\Gamma$ and having being ordered as to make $a_1<a_2<\cdots<a_K$. Furthermore, since $\varepsilon<\rho(I^+)$, there is some neighbourhood $U$ of $\sup\tilde{I}=\frac{c+d}{2}$ which doesn't intersect $E(\varepsilon)$, so $b_K<\sup\tilde{I}$ and $d(b_j,E)\geq\varepsilon$ for every $1\leq j\leq K$. It follows that the $(a_j,b_j)$ are pairwise disjoint and $b_j-a_j\geq2\varepsilon$ unless, perhaps, for $j=1$. Now, take $\varepsilon':=\frac{\gamma}{4}\varepsilon$. We have that
    \begin{align*}
        |F(\varepsilon)\setminus F(\varepsilon')| =\sum_{j=1}^K|(a_j,b_j)\setminus E(\varepsilon')|.
    \end{align*}
    If $b_1-a_1<2\varepsilon$, then $a_1=a$ and since $d(b_1,E)\geq\varepsilon$, it follows that $(b_1+\varepsilon'-\varepsilon,b_1)=(b_1+\varepsilon[\frac{\gamma}{4}-1],b_1)$ does not intersect with $E(\varepsilon')$. Let us observe that in the case that $b_1+\varepsilon[\frac{\gamma}{4}-1]\leq a$, we must have that $|(a_1,b_1)\setminus E(\varepsilon')|=|(a_1,b_1)|$. On the other hand, if $b_1+\varepsilon[\frac{\gamma}{4}-1]>a$, we can estimate
    $$|(a_1,b_1)\setminus E(\varepsilon')|\geq|(b_1+\varepsilon[\tfrac{\gamma}{4}-1],b_1)|=\varepsilon[1-\tfrac{\gamma}{4}]\frac{|(a_1,b_1)|}{|(a_1,b_1)|}\geq\frac{3}{8}(b_1-a_1).$$
    We now estimate $|(a_j,b_j)\setminus E(\varepsilon')|$ for $1\leq j\leq K$ in the general case where $b_j-a_j\geq2\varepsilon$. Set $I_j$ as the open interval for which $I_j^-=(a_j,b_j)$. Recalling once again that $d(b_j,E)\geq\varepsilon$, we find that $(b_j,b_j+\varepsilon)\subset I^+_j$ and $(b_j,b_j+\varepsilon)\cap E=\emptyset$. This implies that $\rho(I^+_j)\geq\frac{1}{2}\varepsilon$ so, because $E$ is $(\sigma,\gamma,+)$-w.p., there exists a collection of pairwise disjoint open intervals $\{J^j_i\}_{i=1}^N$, each contained in $I_j^-\setminus E$, such that $|J^j_i|\geq2\gamma\rho(I^+_j)\geq\gamma\varepsilon$ and $\sum_{i=1}^N|J^j_i|\geq\sigma|I^-_j|$. Notice that $|J^j_i\setminus E(\varepsilon')|\geq\frac{1}{2}|J^j_i|$. In consequence,
    \begin{align*}
        |(a_j,b_j)\setminus E(\varepsilon')|\geq\sum_{i=1}^N|J^j_i\setminus E(\varepsilon')|\geq\frac{1}{2}\sum_{i=1}^N|J^j_i|\geq\frac{\sigma}{2}|I^-_j|=\frac{\sigma}{2}(b_j-a_j).
    \end{align*}
    From this, we finally get
    $$|F(\varepsilon)\setminus F(\varepsilon')|\geq\min\Big\{\frac{3}{8},\frac{\sigma}{2}\Big\}\sum_{j=1}^K|(a_j,b_j)|=:(1-\beta_2)|F(\varepsilon)|.$$
    The statement now follows by setting $\beta_1:=\frac{\gamma}{4}$.
\end{proof}

The next is an interesting result concerning the Hausdorff dimension $\text{dim}_\text{H}E$ of a set $E$ satisfying Definition~\ref{def3.3}. 

\begin{corollary}
    If $E$ is a right-sided weakly porous set, then $\text{dim}_\text{H}\:E\leq1-\frac{\log\beta_2}{\log\beta_1}$, where $\beta_1$ and $\beta_2$ are the constants appearing in Lemma~\ref{lemma4.2}.
\end{corollary}

\begin{proof}
    Fix an open interval $I$ and consider the function $\phi_I:(0,1)\to\mathbb{R}$ given by $\phi_I(t)=|F(t\rho(I^+))|$, where $F(\varepsilon)=E(\varepsilon)\cap\tilde{I}$ as before. Then, $\phi_I$ is non-decreasing, $\phi_I(t)\leq|I|$, and Lemma~\ref{lemma4.2} tell us that $\phi_I(\beta_1t)\leq\beta_2\phi_I(t)$ for every $t\in(0,1)$. Fixed $t$, choose $m$ as the integer such that $\frac{t}{\beta_1^m}<1\leq\frac{t}{\beta_1^{m+1}}$ (if and only if $m<\frac{\log t}{\log \beta_1}\leq m+1$). Then, it follows that $\phi_I(t)=\phi_I\Big(\beta_1^m\frac{t}{\beta_1^m}\Big)\leq\beta_2^m\phi_I\Big(\frac{t}{\beta_1^m}\Big)\leq\beta_2^m|I|\leq\frac{1}{\beta_2}t^{\frac{\log\beta_2}{\log\beta_1}}|I|$. Observing that $|F(\varepsilon)|=\phi_I(\frac{\varepsilon}{\rho(I^+)})$, we get
    \begin{equation}
    \label{eq4.3}
        |F(\varepsilon)|\leq \frac{|I|}{\beta_2}\Big(\frac{\varepsilon}{\rho(I^+)}\Big)^{\alpha_0},
    \end{equation}
    for every $0<\varepsilon<\rho(I^+)$ and setting $\alpha_0=\frac{\log\beta_2}{\log\beta_1}$. We can obtain a bound on the Hausdorff dimension of $E$ using this in the following manner. Considering the intersection of the set $E$ with some arbitrary interval $J$, it is known that $\text{dim}_H(E\cap J)\leq\overline{\text{dim}}_M(E\cap J)$, where $\overline{\text{dim}}_M$ refers to the upper Minkowski dimension. For a bounded set $\Omega\subset\mathbb{R}$, this dimension can be estimated by the formula $\overline{\text{dim}}_M\Omega =\sup\big\{s>0:\limsup_{\varepsilon\to0}\frac{|\Omega(\varepsilon)|}{\varepsilon^{1-s}}>0\big\}$, see \cite{MATTILA}. In order to apply \eqref{eq4.3} and estimate the Minkowski dimension of $E\cap J$, suppose $J$ has endpoints $a<b$ and set $I$ as an open interval such that $I^-=(a-1,b+1)$. Also, let $(c,d)$ be an interval contained in $I^+$ such that $d-c=2\rho(I^+)$. Then, we have that $(E\cap J)(\varepsilon)\subset E(\varepsilon)\cap(a-1,\frac{c+d}{2})=E(\varepsilon)\cap\tilde{I}=F(\varepsilon)$ for every $0<\varepsilon<1$. Now, for every $s>1-\alpha_0$ and $0<\varepsilon<\min\{1,\rho(I^+)\}$,
    \begin{align*}
        \frac{|(E\cap J)(\varepsilon)|}{\varepsilon^{1-s}}\leq\frac{|F(\varepsilon)|}{\varepsilon^{1-s}}\leq C(\sigma,\gamma,I)\:\varepsilon^{s+\alpha_0-1}\to0\text{ as }\varepsilon\to0.
    \end{align*}
    This shows that $\overline{\text{dim}}_M(E\cap J)\leq1-\alpha_0$ for every interval $J$. Thus, $\text{dim}_HE\leq1-\alpha_0$.
\end{proof}

\begin{theorem}
    \label{theorem4.1}
    Being $E$ a non-empty subset of $\mathbb{R}$, the following statements are equivalent.
    \begin{itemize}
        \item[I.] There is some $\alpha>0$ such that $d(\cdot,E)^{-\alpha}\in A_1^+(\mathbb{R})\cap L^1_\text{loc}(\mathbb{R})$.  
        \item[II.] $E$ is right-sided w.p.
    \end{itemize}
\end{theorem}

\begin{proof}
    \noindent(I$\implies$II) By Proposition~\ref{prop3.1} \textit{i.} and the fact that $d(\cdot,E)=d(\cdot,\Bar{E})$ there is no loss of generality by assuming that $E$ is a closed set. Since $w(x):=d(x,E)^{-\alpha}\in L^1_\text{loc}(\mathbb{R})$ and $w_{|_E}=\infty$, it follows that $|E|=0$. Let $\gamma$ be a number in $(0,1)$ to be chosen later and, fixed some open interval $I$, write
    $$I^-\setminus E=\bigcup_{j\in\Gamma}(a_j,b_j),$$
    where the intervals $(a_j,b_j)$ are pairwise disjoint. Next, we set $\Omega:=\{y\in I^-\setminus E:y\in(a_j,b_j)\text{ with }b_j-a_j<2\gamma\rho(I^+)\}$. Notice that $E$ satisfies the right-sided w.p. condition on $I$ with constants $\sigma,\gamma\in(0,1)$ if and only if $|(I^-\setminus E)\setminus\Omega|\geq\sigma|I^-|$, so it suffices to verify this bound for an appropriate value of $\sigma$.

    If $\Omega=\emptyset$, then $|(I^-\setminus E)\setminus\Omega|=|I^-|$ and the required inequality holds for any choice of $\sigma\in(0,1)$. If instead $\Omega\neq\emptyset$, take $x\in\Omega$ to be an arbitrary point and let $j\in\Gamma$ be the index such that $(a_j,b_j)$ contains $x$. Then, $(a_j,b_j)\neq I^-$, since
    $$b_j-a_j<2\gamma\rho(I^+)\leq\gamma|I^+|=\gamma|I^-|.$$
    It follows that either $a_j\in E$ or $b_j\in E$. Therefore, $d(x,E)\leq b_j-a_j<2\gamma \rho(I^+)$, or equivalently $\rho(I^+)^{-\alpha}<(2\gamma)^\alpha d(x,E)^{-\alpha}$, and this holds true for every $x\in\Omega$. Starting from this and integrating on $\Omega$ we get
    \begin{align*}
        \rho(I^+)^{-\alpha}\frac{|\Omega|}{|I^-|} < \frac{(2\gamma)^\alpha}{|I^-|} \int_\Omega d(x,E)^{-\alpha}dx&\leq \frac{(2\gamma)^\alpha}{|I^-|} \int_{I^-} d(x,E)^{-\alpha}dx \\
        &\leq  (2\gamma)^\alpha[w]_{A_1^+(\mathbb{R})}\:\text{ess inf}_{x\in I^+}\:d(x,E)^{-\alpha} \\
        &\leq  (2\gamma)^\alpha[w]_{A_1^+(\mathbb{R})}\:d(z,E)^{-\alpha} \\
        &= C(\alpha) \gamma^\alpha\:\rho(I^+)^{-\alpha},
    \end{align*}

    \noindent where $z$ denotes the center of some interval $J\subset I^+\setminus E$ with $|J|=2\rho(I^+)$. Simplifying and recalling that $|E|=0$, we find
    $$|\Omega|<C(\alpha)\gamma^\alpha|I^-|\iff (1-C(\alpha)\gamma^\alpha)|I^-|<|(I^-\setminus E)\setminus\Omega|,$$
    \noindent so it suffices to pick $\gamma\in(0,1)$ sufficiently small as to make $C(\alpha)\gamma^\alpha<1$ and then setting $\sigma=1-C(\alpha)\gamma^\alpha$ to achieve the sought inequality. 

    \noindent(II$\implies$I) Assume now that $E$ is a closed $(\sigma,\gamma,+)$-w.p. set and fix an open interval $I$. We must find constants $\alpha,C>0$, both independent from $I$, such that
    $\fint_{I^-}d(x,E)^{-\alpha}dx\leq C\:\text{inf ess}_{x\in I^+}d(x,E)^{-\alpha}$. Once this inequality is verified, local integrability follows from the fact that $\text{inf ess}_{x\in I^+}d(x,E)^{-\alpha}\leq d(z,E)^{-\alpha}=\rho(I^+)^{-\alpha}<\infty$, where $z$ is the center of some interval $J\subset I^+\setminus E$ with $|J|=2\rho(I^+)$. We split the problem into three cases depending on the nature of the interval $I$.
    \vspace{0.2cm}
    
    \noindent\textbf{Case 1} ($ I^-\cap E=\emptyset\land d(I,E)>2|I|$): for every $x\in I^-$ and $y\in I^+$ the triangular inequality implies that $d(y,E)-|x-y|\leq d(x,E)$. From this, we find that
    \begin{align*}
        d(x,E) \geq d(y,E)-|x-y|\geq d(y,E)-|I|\geq d(y,E)-\frac{1}{2} d(I^-,E) \geq \frac{1}{2} d(y,E).
    \end{align*}
    For $x\in I^-$ fixed, this means that $d(x,E)^{-\alpha}\leq 2^\alpha\:\text{ess inf}_{y\in I^+}d(y,E)^{-\alpha}$ for any $\alpha>0$ and thus
    $$\fint_{I^-}d(x,E)^{-\alpha}dx\leq 2^\alpha\:\text{ess inf}_{y\in I^+}d(y,E)^{-\alpha}.$$

    \noindent\textbf{Case 2} ($ I^-\cap E=\emptyset\land d(I,E)\leq2|I|$): write $I^-=(a,b)$ and begin by making the following estimates assuming $0<\alpha<1$.
    \begin{align*}
        \fint_{I^-}d(x,E)^{-\alpha}dx \leq \frac{1}{b-a}\int_a^bd(x,\partial I^-)^{-\alpha}dx &=\frac{2}{b-a}\int_a^{\frac{a+b}{2}}(x-a)^{-\alpha}dx \\
        &= \frac{2}{b-a}\int_0^{\frac{b-a}{2}}u^{-\alpha}du= (1-\alpha)\Big(\frac{b-a}{2}\Big)^{-\alpha}.
    \end{align*}
    Since $b-a=|I^-|=|I^+|\geq2\rho(I^+)$, we have
    \begin{equation}
        \label{eq4.2}
        \fint_{I^-}d(x,E)^{-\alpha}dx\leq(1-\alpha)\Big(\frac{b-a}{2}\Big)^{-\alpha}\leq C(\alpha)\:\rho(I^+)^{-\alpha}.
    \end{equation}
    The $A_1^+$ condition in this case now follows from Lemma~\ref{lemma4.1}. \\
    
    \noindent\textbf{Case 3} ($ I^-\cap E\neq\emptyset$): as in the previous case, it suffices to check that $\fint_{I^-}d(x,E)^{-\alpha}dx\lesssim\rho(I^+)^{-\alpha}$, from which the $A_1^+$ condition follows by applying Lemma~\ref{lemma4.1}. To get this bound, let us start by considering the collection of sets $\{F(\varepsilon)\}_{\varepsilon>0}$ introduced in Lemma~\ref{lemma4.2} and take $0<\alpha<1$ to be chosen later. Then,
    \begin{align*}
        \fint_{I^-}d(x,E)^{-\alpha}dx&\leq\frac{1}{|I^-|}\Big(\int_{F(\frac{1}{2}\rho^+)}d(x,E)^{-\alpha}dx+\int_{I^-\setminus F(\frac{1}{2}\rho^+)}d(x,E)^{-\alpha}dx\Big) \\
        &\leq\frac{1}{|I^-|}\int_{F(\frac{1}{2}\rho^+)}d(x,E)^{-\alpha}dx+\Big(\frac{\rho^+}{2}\Big)^{-\alpha},
    \end{align*}
    where $\rho^+:=\rho(I^+)$. For the remaining integral  in the last inequality we make
    \begin{align*}
        \int_{F(\frac{1}{2}\rho^+)}d(x,E)^{-\alpha}dx &=\sum_{k=0}^\infty\int_{F(\frac{1}{2}\beta_1^k\rho^+)\setminus F(\frac{1}{2}\beta_1^{k+1}\rho^+)}d(x,E)^{-\alpha}dx \\
        &\leq\Big(\frac{\beta_1}{2}\rho^+\Big)^{-\alpha}\sum_{k=0}^\infty\beta_1^{-\alpha k}|F(\tfrac{1}{2}\beta_1^k\rho^+)| \\
        &\leq\Big(\frac{\beta_1}{2}\rho^+\Big)^{-\alpha}|F(\tfrac{1}{2}\rho^+)|\sum_{k=0}^\infty(\beta_1^{-\alpha}\beta_2)^k \\
        &\leq C(\alpha)|I|\rho(I^+)^{-\alpha},
    \end{align*}
    where $\alpha>0$ was chosen sufficiently small as to make $\beta_1^{-\alpha}\beta_2<1$. We have shown the existence of constants $\alpha$ and $C$ such that $\fint_{I^-}d(x,E)^{-\alpha}dx\leq C\:\rho(I^+)^{-\alpha}$ for the remaining case, which proves the $A_1^+$ condition of the weight $d(\cdot,E)^{-\alpha}$.
\end{proof}



%
%


\subsection*{Acknowledgements}
This work was supported by Consejo Nacional de Investigaciones Cient\'ificas y T\'ecnicas-CONICET in Argentina. 


\bigskip

\bigskip

%
\noindent{\textit{Affiliation~1.} 
	\textsc{Instituto de Matem\'{a}tica Aplicada del Litoral ``Dra. Eleonor Harboure'', CONICET, UNL.}

	\noindent \textit{Address.} \textmd{IMAL, Streets F.~Leloir and A.P.~Calder\'on, CCT CONICET Santa Fe, Predio ``Alberto Cassano'', Colectora Ruta Nac.~168 km~0, Paraje El Pozo, S3007ABA Santa Fe, Argentina.}
	
	%
	\noindent \textit{E-mail:} \verb|haimar@santafe-conicet.gov.ar| \\ \hspace*{1.3cm} \verb|ivanagomez@santafe-conicet.gov.ar| \\ \hspace*{1.3cm} \verb|ignaciogomez@santafe-conicet.gov.ar|
}

\medskip
\noindent{\textit{Affiliation~2.} 
	\textsc{Departamento de Análisis Matemático, Estadística e Investigación Operativa y Matemática Aplicada, Facultad de Ciencias, Universidad de Málaga,  Málaga, Spain}

	%
	\noindent \textit{E-mail:} \verb|martin_reyes@uma.es| 
}

\end{document}